\documentclass[a4paper,12pt]{article}

\usepackage{amsthm}
\usepackage{amsmath}
\usepackage{epsfig}
\usepackage{amssymb}
\usepackage{latexsym}
\usepackage[T1]{fontenc}
\usepackage[utf8]{inputenc}
\usepackage{color}

\author{
Robert Luko\v{t}ka\footnote{Corresponding author, Faculty of Mathematics Physics and Informatics, Comenius University, Mlynská dolina F1, 842~48~Bratislava, Slovakia}\\Comenius University, Bratislava\\{\small\tt lukotka\@@dcs.fmph.uniba.sk}\\
Edita Rollov\' a\\University of West Bohemia, Plze\v n\\{\small\tt rollova\@@ntis.zcu.cz}
}
\title{Perfect matchings in highly cyclically connected regular graphs}

\theoremstyle{definition}
\newtheorem*{theorem*}{Theorem}
\newtheorem{definition}{Definition}
\newtheorem{theorem}[definition]{Theorem}

\newtheorem{lemma}[definition]{Lemma}

\newcommand{\oc}{{\text{oc}}}

\let\epsilon=\varepsilon

\let\phi = \varphi

\begin{document}

\maketitle


\abstract{
A leaf matching operation on a graph consists of removing a vertex of degree~$1$ together with its neighbour from the graph. 
Let $G$ be a $d$-regular cyclically $(d-1+2k)$-edge-connected graph of even order,  where $k \ge 0$ and $d\ge 3$. We prove that for any given set $X$ of $d-1+k$ edges, there is no $1$-factor of $G$ avoiding $X$
if and only if either 
an isolated vertex can be obtained by a series of leaf matching operations in $G-X$,
or $G-X$ has an independent set that contains more than half of the vertices of~$G$.
To demonstrate how to check the conditions of the theorem we prove several statements on $2$-factors of cubic graphs.
For $k\ge 3$, we prove that given a cyclically $(4k-5)$-edge-connected cubic graph $G$ and
three paths of length $k$ such that the distance between any two of them 
is at least $8k-16$, there is a $2$-factor of $G$ that contains one of the paths. 
We provide a similar statement for two paths when $k=3$ and $k=4$.
As a corollary we show that given a vertex $v$ in a cyclically $7$-edge-connected 
cubic graph, there is a $2$-factor such that $v$ is in a circuit of length greater than $7$.

\smallskip

}


\section{Introduction}
A \emph{perfect matching} of a graph $G$ is a subset of edges of $G$ such that every vertex of the graph is incident with exactly one edge of the subset.
In this paper, we study perfect matchings of highly connected regular graphs.  It is well-known that if $G$ is a $d$-regular $(d-1)$-edge-connected graph of even order, then $G$ admits a perfect matching~\cite{berge}. Moreover, the following holds according to Plesn\'ik~\cite{plesnik}.

\begin{theorem}\label{plesnik}
Let $G$ be a $d$-regular $(d-1)$-edge-connected graph of even order. If $X$ is a set of edges of $G$ with $|X|\leq d-1$, then $G-X$ admits a perfect matching.
\end{theorem}

Theorem~\ref{plesnik} has been examined in various modifications and generalizations. 
Instead of assuming that all edge-cuts have size at least $d-1$ one needs to assume it only for odd edge-cuts (\cite{cruse}, \cite{lin}). As if we remove an odd edge-cut from the graph the resulting graph does not admit a perfect matching, Cruse's result holds in the converse direction.
Chartrand and Nebeský considered the existence of a $1$-factor in $(d-2)$-edge-connected graphs \cite{nebesky}, while several authors (see e.g. \cite{liu, katernis, shiu}) inspected the existence of a $k$-factor, for $1 \le k \le d-1$, where a \emph{$k$-factor} is a $k$-regular spanning subgraph. Note that a perfect matching is a $1$-factor, and thus its complement in a $d$-regular graph is $(d-1)$-factor. 

The \emph{matching preclusion number} is the minimum number of edges whose deletion 
results in a graph without a perfect matching (or an almost-perfect matching when graph has odd order, but we consider only graphs of even order in this paper). Using this notion Theorem~\ref{plesnik} reads as ``The matching preclusion number of a $d$-regular $(d-1)$-edge-connected graph of even order is $d$\,''. Cheng et~al.~\cite{cheng1} defined the \emph{conditional matching preclusion number} to be the minimum number of edges whose deletion
yields a graph without a perfect matching and an without an isolated vertex. Cheng et~al. proved in \cite{cheng2} that 
if all edge-cuts of size at most $d$ in a $d$-regular bipartite graph $G$ separate a single vertex, then the conditional matching preclusion number of $G$ is at least $d+1$. The same statement is not true if we omit the biparticity condition \cite{cheng3}. However if we add additional condition that $G$ has no independent set of size $|V(G)|/2-1$, then the conditional matching preclusion number of $G$ is at least $d+1$ \cite[Theorem 2.3]{cheng3}. Lin and Zhang \cite[Theorems~3.2 and~3.5]{lin} proved that these two conditions (cuts on the size of independent set) are not only necessary, but also sufficient.

Besides that, Cheng et al. \cite{cheng3} defined that a graph is \emph{super-$k$-edge-connected of order $l$} if every edge-cut of size at most $k$ separates one large component from the remaining components that together contain at most $l$ edges. In \cite[Theorems~3.1-3.4]{cheng3} Cheng et al. show that super-$(3k-6)$, $(3k-7)$ and $(3k-8)$-edge-connected graphs of order $2$ have the conditional matching preclusion number equal to $2k-3$ under various assumptions (i. e. on the size of the largest independent set).

In this paper, we present strengthening of Theorem~\ref{plesnik} for graphs of higher connectivity. A major obstacle in obtaining such a result is 
the fact that maximum edge-connectivity in a $d$-regular graph is $d$. In order to extend the notion of edge-connectivity in a $d$-regular graph beyond $d$, we consider a cyclic edge-connectivity, which is a parameter commonly used in the related area of matching extensions (see e. g. \cite{aldred, holton, lou}). A graph is \emph{cyclically $h$-edge-connected} 
if there is no edge-cut containing less than $h$ edges separating
two subgraphs both containing a circuit. If $G$ is $d$-regular
and $h\le d$, then cyclic $h$-edge-connectivity coincides with
$h$-edge-connectivity. Therefore, cyclic edge-connectivity is a natural extension of the notion of 
edge-connectivity. Moreover, if a $d$-regular graph is cyclically $(d+1)$-edge-connected, for $d\geq 3$, then every $d$-edge-cut must have a single vertex on at least one side. 
As the main result we prove that if we raise the cyclic edge-connectivity of $G$ by $2k$ compared to the requirements of Theorem~\ref{plesnik}, then the set $X$ can contain $k$ extra edges, provided two obvious necessary conditions are satisfied.

Let $H$ be a graph with a vertex $v$ of degree $1$. Let $w$
be the neighbour of $v$ in $H$. Delete the vertices $v$ and $w$ to obtain $H'$. The graph $H'$ has a perfect matching if and only if $H$ has a perfect matching. The operation of deleting $v$ and $w$ from $H$ is called a \emph{leaf matching operation} or an \emph{LM operation}, for short. 
The following theorem is the main result of this paper.

\begin{theorem}\label{thm1}
Let $G$ be a $d$-regular cyclically $(d-1+2k)$-edge-connected graph
of even order, for  $k\ge 0$. Let $X$ be a set of $d-1+k$ edges of $G$. 
The graph $G-X$ does not admit a perfect matching
if and only if 
\begin{enumerate}
\item[\emph{(i)}] an isolated vertex is obtained in $G-X$ by a series of leaf matching operations, or
\item[\emph{(ii)}] $G-X$ has an independent set containing more than half of the vertices of $G$. 
\end{enumerate} 
Moreover, if (ii) holds, then the vertices that do not belong to the independent set induce at most $k-1$ edges.
\end{theorem}

\noindent Note that the existence of an independent set from (ii) of Theorem~\ref{thm1} would imply that $G$ is close to being a bipartite graph with uneven sizes of partite sets. 

For $k=0$, Theorem~\ref{thm1} coincides with Theorem~\ref{plesnik} as neither (i) nor (ii) is possible. Indeed, for (i), it can be proved by induction that the edge-cut whose one side has vertices participating in the LM operations contains at most $2d-2$ edges.
This is not possible as the other part has even number of vertices (for more details see the proof of Theorem~\ref{thm3} which uses similar ideas). For (ii), the moreover part of the theorem would be false if (ii) was true. For $k=1$, we are able to reformulate Theorem~\ref{thm1} as follows. This result is closely related to Theorem 3.5 in \cite{lin}.

\begin{theorem}\label{thm3}
Let $G$ be a $d$-regular cyclically $(d+1)$-edge-connected graph of even order, for $d\geq 3$.
Let $X$ be a set of $d$ edges of $G$. 
The graph $G-X$ does not admit a perfect matching
if and only if
\begin{enumerate}
\item[\emph{(a)}] \label{c22} all the edges from $X$ are incident with a common vertex, or
\item[\emph{(b)}] \label{c12} $G-X$ is bipartite and the edges from $X$ are incident with vertices from the same partite. 
\end{enumerate} 
\end{theorem}

\begin{figure}[h]
\begin{center}
\includegraphics[width=8cm]{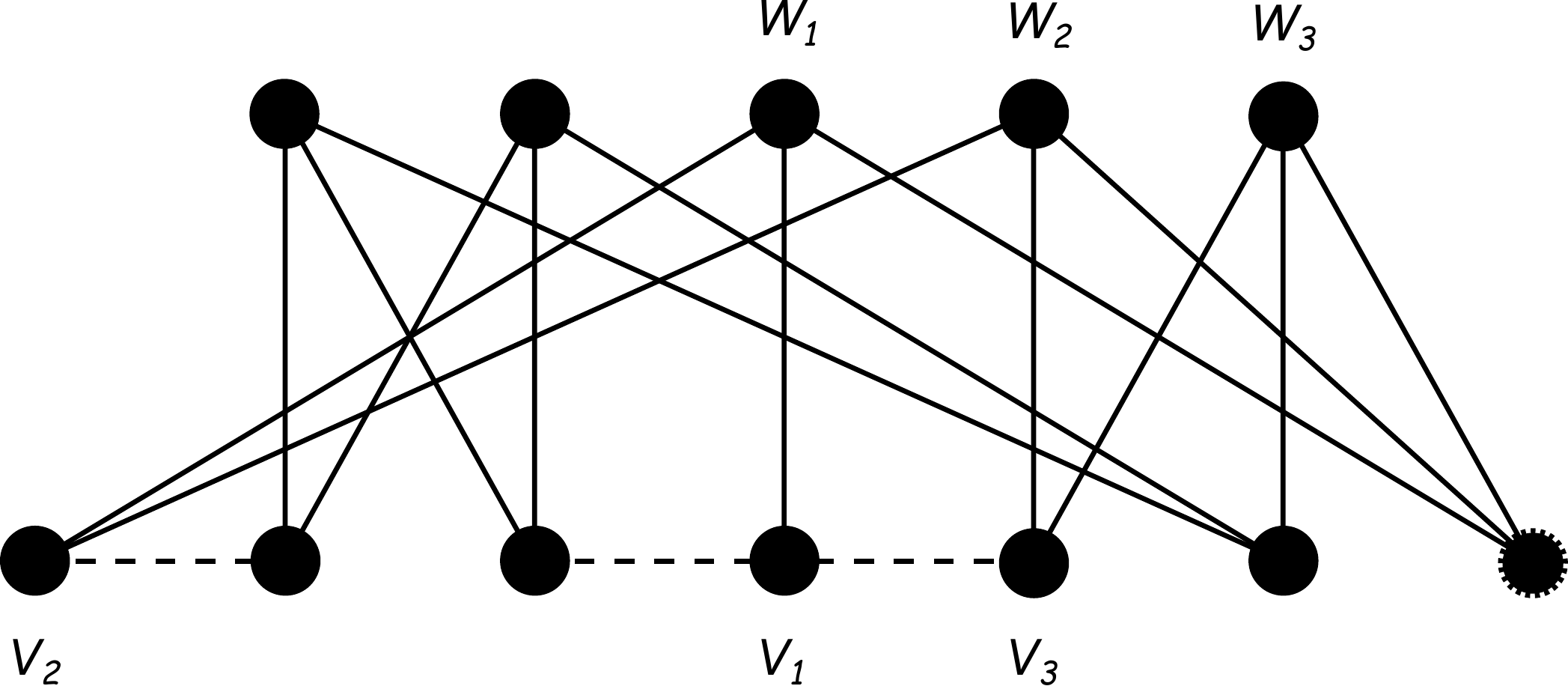}
\caption{Cyclically $4$-edge-connected cubic graph without an $1$-factor avoiding dashed edges}
\label{fig1}
\end{center}
\end{figure}
It is interesting that the two conditions of Theorem~\ref{thm1} and Theorem~\ref{thm3} are not in one-to-one correspondence for every graph.  Figure~\ref{fig1} provides an example of a graph $G$, where the set $X$ is depicted by dashed edges. Note that while the edges of $X$ are not all incident with a common vertex, an isolated vertex can be obtained in $G-X$ by a series of LM operations (the subscript of the vertices indicates the order of LM operations applied in $G-X$).

\bigskip

Let $G$ be a graph and let $X$ be a set of edges from $G$. If $G$ is sufficiently high cyclically edge-connected, then the condition (i) from Theorem~\ref{thm1} is a local condition, which can be often verified easily. Thus, in order to use Theorem~\ref{thm1}, it is crucial to have a method that allows us to verify the condition (ii) efficiently. We provide one such method while we prove several theorems
on \emph{cubic} graphs, which are $3$-regular graphs. Recall that the complement of a perfect matching in a cubic graph is a $2$-factor.
Many important theorems and conjectures can be reformulated in terms of $2$-factors of cubic graphs. For instance, the $4$-colour theorem can be restated as ``Every bridgeless planar cubic graph has a $2$-factor without circuits of odd lengths''. Considering general cubic graphs of high cyclic edge-connectivity, very interesting conjecture of Jaeger and Swart \cite{jaeger} states that every cyclically $7$-edge-connected cubic graph is $3$-edge-colourable, i.~e.~has a $2$-factor without odd circuits. A solution to this conjecture would shed more light on several important conjectures, e.~g.~to prove the $5$-flow conjecture it would be sufficient to consider cubic graphs with cyclic edge-connectivity equal to $6$ \cite{kochol}. However, the progress towards proving the conjecture of Jaeger and Swart is very limited.

Aldred et al. \cite{aldred2} examined when it is possible to extend a matching into a $2$-factor. Here, we aim to extend a path $P$ of a given cubic graph into a $2$-factor. This is possible if and only if $G-E(P)$ admits a perfect matching. Note that according to Theorem~\ref{plesnik}, it is always possible to find a perfect matching that avoids a path of length $1$ or $2$ (where the \emph{length of a path} is the number of its edges).  For longer paths, Theorem~\ref{thm1} gives a necessary and sufficient condition when the extension is possible. 
Let $G$ be a graph, and let $G_1$ and $G_2$ be two subgraphs of $G$. We define the \emph{distance} of $G_1$ and $G_2$ to be $\min\{ d(v_1,v_2) | v_1\in G_1,v_2\in G_2 \}$, where $d(v_1,v_2)$ is the length of a shortest path between $v_1$ and $v_2$ in $G$. The distance between the prescribed edges is a common prerequisite in theorems on matching extensions (see e.~g.~\cite{aldred}).
We apply Theorem~\ref{thm1} to show, that if we have three paths on $l$ edges that are sufficiently far apart, for a given integer $l$ in a sufficiently cyclically edge-connected graph, one of the paths can be extended into a $2$-factor.

\begin{theorem}\label{thm4}
Let $G$ be a cyclically $\max\{4l-5,2\}$-edge-connected cubic graph and let $P_1$, $P_2$, and $P_3$ be
three paths of $G$ on $l$ edges such that the distance of any two of them 
is at least $\max\{8l-16,0\}$. Then there is a $2$-factor of $G$ that contains 
$P_1$, $P_2$ or $P_3$.
\end{theorem} 

We note that while the cyclic edge-connectivity and the path distance from Theorem~\ref{thm4} may not need to be optimal, the condition on the number of paths cannot be reduced. Let us demonstrate that Theorem~\ref{thm4} does not hold if only two paths are prescribed, even if we assume arbitrary cyclic edge-connectivity and distance between the paths. Let $H$ be a bipartite subcubic graph with partites $M$ and $N$ such that $M$ contains exactly $18$ vertices of degree $2$ and a couple of vertices of degree $3$, while 
$N$ contains only vertices of degree $3$. Partition the set of vertices of degree $2$ into two sets $A=\{a_1, \ldots, a_9\}$ and $B=\{b_1,\ldots, b_9\}$ of size $9$. Add new vertices $x_i$ and $y_i$ to $H$, for $i\in\{ 1,\ldots, 7\}$. Add new edges $a_1x_1$, $a_8x_7$ $a_9x_7$, $b_1y_1$, $b_8y_7$, $b_9y_7$, and for $i\in\{1,\ldots, 6\}$, add edges $a_{i+1}x_i$, $b_{i+1}y_i$, $x_{i+1}x_i$ and $y_{i+1}y_i$. Note that the resulting graph is cubic (see Figure~\ref{figthm4}).
\begin{figure}
\begin{center} 
\includegraphics[scale=1.2]{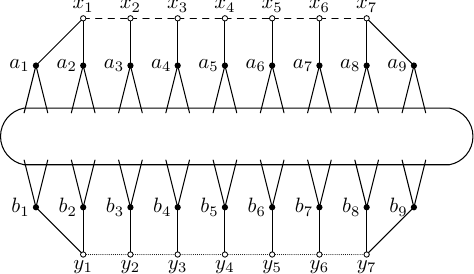} \\
\caption{Construction of a graph containing two $6$-paths that are arbitrarily far apart, but neither of them can be extended into a $2$-factor.}
\label{figthm4}
\end{center}
\end{figure}

Set $P_1=x_1x_2\cdots x_7$ and $P_2=y_1y_2\cdots y_7$.
The edge-set of $P_1$ cannot be extended into a $2$-factor, because $\{x_1, x_2, \ldots, x_7\} \cup N \cup \{y_1, y_3, y_5, y_7\}$ is an independent set of $G-E(P_1)$ containing more than half of the vertices.
Similarly, the edge-set of $P_2$ cannot be extended into a $2$-factor, because
$\{y_1, y_2, \ldots, y_7\} \cup N \cup \{x_1, x_3, x_5, x_7\}$ is an independent set of $G-E(P_2)$ containing more than half of the vertices. Note that the choice of $H$ can be done so that the resulting graph is arbitrarily cyclically edge-connected, and $P_1$ and $P_2$ are in any arbitrary distance.

On the other hand, if the paths have fewer than $6$ edges, then it is possible to prove analogous result to Theorem~\ref{thm4} even for two paths. We analyse the cases when $l=3$ and $l=4$ and provide the complete characterisation when  none of these two paths can be extended into a $2$-factor. The following theorem is a corollary of the case when $l=4$.

\begin{theorem} \label{cor7}
Let $G$ be a cyclically $7$-edge-connected cubic graph and let $v\in V(G)$. Then there exists a $2$-factor of $G$ such that $v$ does not belong to its $7$-circuit.
\end{theorem}
Theorem~\ref{cor7} can be viewed as a strongly weakened version of the conjecture of Jaeger and Swart \cite{jaeger}: we do not forbid circuits of all odd lengths, only of the length $7$ and we forbid it only for the circuit passing through a given vertex $v$. Despite the apparent simplicity of   
Theorem~\ref{cor7}, we are not aware of any significantly simpler approach to proving it than the one used in this paper. 
Moreover, we are not aware of any result that would imply a version of
Theorem~\ref{cor7} where all occurrences of the number $7$ is replaced by a larger integer.

The organisation of the paper is as follows. In Section~\ref{s2} we prove Theorem~\ref{thm1} and Theorem~\ref{thm3}. In Section~\ref{s3} we prove Theorem~\ref{thm4}. In Section~\ref{s4} we deal with extensions of two paths of the length 3 and 4 into a $2$-factor, and we prove Theorem~\ref{cor7}.


\section{Perfect matchings without preselected edges}\label{s2}

For a graph $G$ and a vertex $v\in V(G)$, let $\deg_G(v)$ denote the degree
of $v$ in $G$. For a set $W\subseteq V(G)$ let $\partial_G(W)$ denote the set of edges of $G$ with exactly one end in~$W$. For a subgraph $H$ of a graph $G$, we abbreviate $\partial_G(V(H))$ to $\partial_G(H)$. 
Let $\oc(G)$ denote the number of odd components of $G$. 
For two sets $A$ and $B$, $A\Delta B$ denotes the \emph{symmetric difference} of $A$ and $B$, i. e. $A\Delta B = (A-B) \cup (B-A)$.
Now we proceed with a proof of Theorem~\ref{thm1}.

\begin{proof}[Proof of Theorem~\ref{thm1}]
If either of the conditions (i) and (ii) is satisfied, then $G-X$ admits no perfect matching.
Therefore, we need to prove only the converse statement. 

Let $G'=G-X$ and suppose that $G'$ has no perfect matching.
Due to Tutte's perfect matching theorem \cite[Theorem 2.2.1]{diestel} there is a vertex-set $S$, such that 
$\oc(G'-S)>|S|$. Since $G'$ has even number of vertices, we conclude that $\oc(G'-S)-|S|\ge 2$.
We choose $S$ so that $\oc(G'-S)-|S|$ is maximal and subject to this condition 
$|S|$ is maximal. We distinguish three cases.
\medskip

\noindent
\underline{Case 1:} There exists a component $H$ of $G'-S$ such that $G'-H$ is acyclic.
\smallskip

\noindent
We prove that an isolated vertex can be produced  by a series of LM operations.
We prove the statement by induction. 
As the induction hypothesis is different from the statement of Theorem~\ref{thm1}
we formulate the statement in a separate lemma.

\begin{lemma}\label{lemaD}
Let $G'$ be a graph, let $S$ be a subset of $V(G')$, and let $H$ be a component of
$G'-S$. Suppose that
\begin{itemize}
\item $G'-H$ is acyclic, and
\item $\oc(G'-S)-|S| \ge 2$.
\end{itemize}   
Then an isolated vertex can be created in $G'$ by a series of LM operations.
Moreover, all such LM operations remove only vertices from $G'-H$
and the  created isolated vertex is also from $G'-H$.
\end{lemma}

\begin{proof}
We proceed by induction on the number of vertices in $G'$.

\noindent
We have $\oc(G'-H-S)-|S| \ge 1$. 
Let $D_1$, $D_2$, \dots, $D_c$ be the components of $G'-H$.
Then
$$
\sum_{i=1}^c \left( \oc(D_i-S)-|D_i \cap S| \right) = \oc(G'-H-S)-|S| \ge 1.
$$
Therefore, there is a component $D^*$ of $G'-H$ such that
\begin{eqnarray}
\oc(D^*-S)-|D^* \cap S| \ge 1. \label{dstar}
\end{eqnarray}
If $D^*$ is a single vertex, then by (\ref{dstar}) this vertex is not in $S$.
Then it is an isolated vertex in $G'$ from $G'-H$, proving the lemma. 
Therefore, $D^*$ has at least two vertices.

Let $d_i$ be the number of vertices from $D^* \cap S$
of degree $i$ in $D^*$. 
As $D^*$ is a tree, the number of components of $D^*-S$
is at most $1+\sum_{i=2}^d (i-1)\cdot d_i$. On the other hand 
$D^*$ has at least $2+\sum_{i=3}^d (i-2)\cdot d_i$ leaves. 
Suppose, for the contradiction, that all leaves of $D^*$ are in $S$. 
Then $D^*\cap S$ has at least
$2+\sum_{i=2}^d (i-2)\cdot d_i+ \sum_{i=2}^d d_i$ vertices. It follows that
$$
\oc(D^*-S)-|D^* \cap S| \le \left(1+\sum_{i=2}^d (i-1)\cdot d_i\right)
 - \left(2+\sum_{i=2}^d (i-1)\cdot d_i\right) = -1,
$$
which contradicts (\ref{dstar}). Therefore $D^*-S$ contains a leaf $v$ of 
$D^*$.
  
Let $w$ be the neighbour of $v$.
We set  $G''=G'-\{v, w\}$, $S'=S-\{w\}$ (note that $S$ may not contain $w$)
and $H'=H$. The prerequisites of the lemma are satisfied for $G''$, $S'$ and $H'$ taken as $G'$, $S$ and $H$, respectively. By the induction hypothesis
we can create an isolated vertex by LM operations in $G''$. 
On the other hand, we can create $G''$ from $G'$ by an LM operation because 
$1=\deg_{D^*}(v)=\deg_{G'}(v)$. 
\end{proof}
\smallskip

\noindent
\underline{Case 2:} The graph $G'-H'$ contains a circuit for every component $H'$ of $G'-S$ and there exists at least one component $H$ of $G'-S$ that contains a circuit.
\smallskip

\noindent
Each of the graphs $H$ and $G'-H$ contains a circuit, hence by the cyclic edge-connectivity assumption $|\partial_G(H)|\ge d-1+2k$.
For every other odd component $H'$ of $G'-S$ we have $|\partial_G(H')|\ge d$ since $G$ is $d$-regular (this holds even when $k=0$, since if $d$ is even, all such edge-cuts are even and if $d$ is odd, then an edge-cut separating subgraphs with odd number of vertices must be odd). 
Let ${\cal H}$ be the set consisting of odd components of $G'-S$ and $H$. It follows that
$$
\sum_{H' \in {\cal H}}  |\partial_G(H')| \ge d \cdot (\oc(G'-S)-1)+d-1+2k.
$$
This translates to
$$
\sum_{H' \in {\cal H}}  |\partial_{G'}(H')| \ge d \cdot (\oc(G'-S)-1)+d-1+2k-2(d-1+k).
$$
Since all these edges have the second end-vertex in $S$ and $|\partial_{G'}(S)| \le d |S|$, it follows that
$$
d \cdot \oc(G'-S) - 2d+1 \le d|S|.
$$
Thus $\oc(G'-S) < |S|+2$, which contradicts $\oc(G'-S)-|S|\ge 2$.
\medskip

\noindent
\underline{Case 3:} The graph $G'-H'$ contains a circuit for every component $H'$ of $G'-S$ and each component of $G'-S$ is a tree.
\smallskip

\noindent
Suppose first that one of the components of $G'-S$, say $H$, contains more than one vertex, and let $w$ be a neighbour of a leaf $u$ of $H$. Let $S'=S\cup\{w\}$. If $H$ is even, then $\oc(G'-S')\geq \oc(G'-S)+1$, since at least one new odd component is created (the one that contains $u$). If $H$ is odd, then again $\oc(G'-S')\geq \oc(G'-S)+1$, because at least two new odd components are created (one of them contains $u$, the other one exists since $|H-\{u,w\}|$ is odd). If $\oc(G'-S')\geq \oc(G'-S)+2$, then we obtain a contradiction with the maximality of $\oc(G'-S)-|S|$. Otherwise, we obtain a contradiction with the maximality of $|S|$.

Suppose now that each component of $G'-S$ is an isolated vertex. Since $|S|<\oc(G'-S)$, the vertices
from $G'-S$ form an independent set containing more than half of the vertices of $G'$, satisfying the condition (ii) of the theorem.

We are left to prove that the set $S$ induces at most $k-1$ edges.
Since $\oc(G'-S)-|S|\ge 2$, it follows that
$$
\sum_{v\in G'-S} \deg_G(v)=d\cdot\oc(G'-S)\geq 2d +d|S|.
$$
In $G'$, this is
$$
\sum_{v\in G'-S} \deg_{G'}(v) \ge 2d +d|S| - 2(d-1+k),
$$
and all these edges have the second end-vertex in $S$. 
Suppose for the contrary, that the vertices of $S$ induce at least $k$ edges.
Then $|\partial_{G'}(S)|\le d|S|-2k$. Combining the last two inequalities we conclude
$$
2d +d|S| - 2(d-1+k) \le d|S|-2k,
$$
a contradiction.
\end{proof}

We finish the section by proving Theorem~\ref{thm3}.
\begin{proof}[Proof of Theorem~\ref{thm3}]
If either of the conditions (a) and (b) is satisfied, then $G-X$ admits no perfect matching. Therefore, we need to prove only the converse statement. 

Suppose that $G-X$ has no perfect matching. According to Theorem~\ref{thm1}, either we create an isolated vertex in $G-X$ by a series of LM operations or $G-X$ is bipartite and edges from $X$ are incident with vertices from the same partite. If the latter one occurs, we are done. Therefore, we may assume that an isolated vertex is created by a series of LM operations in $G-X$. Suppose that we have obtained an isolated vertex after $s$ LM operations. Let $v_i$ be a vertex of degree 1 that we remove during an $i$-th LM operation. Let $w_i$ be a neighbour of $v_i$ that is deleted together with $v_i$ during the $i$-th LM operation. Let $V_i=\bigcup_{j=1}^i\{v_j\}$, $W_i=\bigcup_{j=1}^i\{w_j\}$ and $U_i=V_i\cup W_i$. 

If $s=0$, then $G-X$ contains an isolated vertex, and the condition (a) holds. 
Assume that $s=1$. Note that in order to create a vertex of degree 1 in $G-X$, exactly $d-1$ edges of $X$ must be incident with the vertex $v_1$. Since $s=1$, deleting $U_1$ from $G-X$ results in a graph that contains an isolated vertex $u$. Thus either 
every edge incident with $u$ has the second end-vertex in $U_1$, 
or  there is exactly one edge $e=uu'$ incident with $u$ that does not have the second end-vertex in $U_1$ and $e\in X$. 
In the former case, since $|\partial_G(U_1)| \le 2d-2$, $|\partial_G(U_1\cup\{u\})| \le d-2$, we obtain a contradiction with the cyclic edge-connectivity of $G$. In the latter case, $|\partial_G(U_1\cup\{u\})| \le (2d-2)-(d-1)+1=d$. If $|\partial_G(U_1\cup\{u\})| \le d-1$ or if $|V(G)-U_1-\{u\}|\geq 2$, then we again obtain a contradiction with the cyclic edge-connectivity of $G$. Thus $|\partial_G(U_1\cup\{u\})|=d$ and $|V(G)-U_1-\{u\}|=1$. It follows that $V(G)=\{v_1,w_1,u,u'\}$. Note that $e$ must be the only edge between $u$ and $u'$, since $u$ is an isolated vertex in $(G-X)-U_1$. Thus, $|\partial_G(\{u,u'\})| =2d-2$, and since $G$ is $d$-regular there exists exactly one edge between $v_1$ and $w_1$, and it does not belong to $X$. Since $v_1$ is a vertex of degree $1$ in $G-X$, it follows that every other edge incident with $v_1$ is from $X$ and has an end-vertex in $\{u,u'\}$. Therefore $G-X$ is a bipartite graph with partites $\{w_1\}$ and $\{v_1,u,u'\}$. Moreover, no edge from $X$ is incident with $w_1$, implying the condition (b) of the theorem. Thus we may assume that $s\geq 2$. 

Under these assumptions we prove the following lemma by induction.

\begin{lemma}\label{LemCor2}
Let $G$ be a $d$-regular cyclically $(d+1)$-edge-connected graph of even order. Let $X$ be a set of $d$ edges of $G$. Assume that a first isolated vertex in $G-X$ is created after $s$ LM operations, where $s\geq 2$. Then 
\begin{itemize}
\item[(i)] every edge of $X$ is incident with a vertex of $V_2$,
\item[(ii)] $|\partial_G(U_t)| \le 2d$ for each $2\leq t\leq s$, and $|\partial_G(U_1)| \le 2d-2$, and
\item[(iii)] the edges incident with the vertices of $V_t$ in $G$ either belong to $X$ or they connect a vertex of $V_t$ with a vertex of $W_t$, for each $1\leq t\leq s$.
\end{itemize}
\end{lemma}

\begin{proof}[Proof of Lemma \ref{LemCor2}.]  
First we prove (i), the second part of (ii) and (iii) for $t=1$. Note that the latter one is obvious, because $w_1$ is the only neighbour of $v_1$ in $G-X$. Since $s\geq 2$, in order to obtain an isolated vertex after two LM operations, there must be exactly $d-1$ edges of $X$ incident with $v_1$. As there is an edge connecting $v_1$ and $w_1$ in $G$, $|\partial_G(U_1)|\leq 2d-2$, implying the second part of (ii). Let $e$ be the edge of $X$ that is not incident with $v_1$. Assume to the contrary with (i) that $e$ is not incident with $v_2$. Since $v_2$ is a vertex of degree $1$ after the first LM operation in $G-X$, exactly $d-1$ edges incident with $v_2$ have the other end-vertex in $U_1$. Since the remaining edge incident with $v_2$ has $w_2$ as the other end-vertex, $|\partial_G(U_1\cup \{v_2\})| \le (2d-2)-(d-1)+1=d$. This is a contradiction with the cyclic edge-connectivity of $G$, because each side of the edge-cut has at least two vertices - one side being $U_1\cup\{v_2\}$ and the other side containing $w_2$ and a resulting isolated vertex. 

We prove the first part of (ii) and (iii) by induction on $t$. For the basis of induction, let $t=2$. Recall that $v_2$ is a vertex of degree 1 in $(G-X)-U_1$ and $e$ is incident with $v_2$. Let us count $|\partial_G(U_2)|$. 
There are at most $2d-2$ edges in $\partial_G(U_1)$, at least $d-2$ out of these are incident with $v_2$. 
As $G$ has an edge between $v_2$ and $w_2$,
there is at most one edge in $\partial_G(U_2)$ incident with $v_2$ and at most $d-1$ edges in $\partial_G(U_2)$
incident with $w_2$. We conclude that $|\partial_G(U_2)|\leq (2d-2)-(d-2)+1+(d-1)=2d$, and (ii) holds. Moreover, every edge incident with a vertex of $V_2$ either belongs to $X$ or is incident with a vertex of $W_2$ as requested by (iii).

For the induction step, let us assume that the lemma holds for every $j<t$, we are going to prove it for $t$.  
Let us count $|\partial_G(U_t)|$. There are at most $2d$ edges in $\partial_G(U_{t-1})$, out of these exactly $d-1$ are incident with $v_{t}$. There are at most $d-1$ edges incident with $w_{t}$ that are not incident with $v_{t}$. Therefore $|\partial_G(U_{t})|\leq 2d-(d-1)+(d-1)$, which proves (ii). 

To prove (iii), we only need to prove the statement for the edges incident with $v_{t}$. Since $v_{t}$ is a vertex of degree 1 in $(G-X)- U_{t-1}$, and $v_{t}w_{t}\in E(G)$, every edge incident with $v_t$ has the other end-vertex in $V_{t-1}\cup W_t$. Suppose that there is an edge $v_{t}v_j$ for some $1\leq j\leq t-1$. But the vertex $v_j$ is adjacent only with the vertices of $W_j$ or incident with edges from $X$ by the induction hypothesis. This is a contradiction, and hence every edge incident with  $v_t$ has the other end-vertex in $W_t$ and (iii) holds as well.
\end{proof}

We are ready to finish the proof of Theorem~\ref{thm3}. Let $u$ be a first isolated vertex, which is created after $s$ steps of LM operations, where $s\geq 2$. By Lemma~\ref{LemCor2} (ii), $|\partial_G(U_{s})|\leq 2d$. Thus $d$ out of these $2d$ edges are incident with $u$. Since $G$ is cyclically $(d+1)$-edge-connected, the other $d$ edges must be adjacent to a single vertex, say $u_2$, which is another isolated vertex of $(G-X)- U_s$. In this case, $G-X$ has an independent set $A$ containing more than half of the vertices of $G$, namely $V_s\cup\{u,u_2\}$. Theorem~\ref{thm1} with $k=1$ implies that $G-X$ is bipartite with partites $A$, and $B=G-A$. Note that $|A|\geq |B|+2$, since $|V(G)|$ is even.

It remains to prove that the edges from $X$ have both ends in $A$. Suppose to the contrary that at most $d-1$ edges of $X$ have both ends in $A$. Then $|\partial_G(A)|\geq d(|B|+2)-2(d-1)$. On the other hand, $|\partial_G(A)|\leq d|B|$. The two inequalities imply that $0\geq 2$, which is a contradiction. Thus (b) holds for $s\geq 2$ as well. This concludes the proof of the theorem.
\end{proof}


\section{Paths contained in $2$-factors}\label{s3}

Let $G$ be a cubic graph.
Let $P=p_0p_1\cdots p_m$ be a path of $G$ of length $m$, $m\ge 3$.
Assume that $G$ is cyclically $(2m-1)$-edge-connected. This, up to several small graphs that do not have a cycle-separating edge-cut, implies that $G$ has no circuit of length less than $2m-1$.
Here, $K_2^3$ is a cubic graph on two vertices and three parallel edges.

\begin{lemma}\label{lemaCyklicka}\cite[Proposition 1]{nedela}
A cyclically $c$-edge connected cubic graph $G$ contains no circuit of length less than $c$ unless $G$ is isomorphic to $K_2^3$, $K_4$ or $K_{3,3}$.
\end{lemma}

Our goal is to extend $P$ to a $2$-factor. In $K_2^3$, $K_4$ or $K_{3,3}$ it is possible to extend every path into a $2$-factor. Thus assume that $G$ is not isomorphic to any of these graphs.
Further, assume that $P$ cannot be extended to a $2$-factor of $G$, hence $G-E(P)$ does not admit a perfect matching. According to Theorem~\ref{thm1} applied with $k=m-2$
a series of LM operations in $G-E(P)$ produces an isolated vertex or $G-E(P)$ has an independent set containing more than half of the vertices. 

The only vertices of degree $1$ in $G-E(P)$ are $p_1, \ldots, p_{m-1}$. 
Let $p'_1, \ldots, p'_{m-1}$ be the neighbours of $p_1, \ldots, p_{m-1}$ outside $P$, respectively.
These vertices are all distinct, otherwise $G$ would have a circuit of length at most $m$,
contradicting Lemma~\ref{lemaCyklicka}. We perform LM operations removing vertices $p_i$ and $p'_i$, 
for $i\in\{1,\dots, m-1\}$. If a vertex $p''$ of degree at most $1$ is created, this would imply that
$G$ has a circuit of length at most $m+2$, which is less than $2m-1$ when $m>3$, contradicting Lemma~\ref{lemaCyklicka}.
Thus no further LM operation can be performed since there are no further vertices of degree $1$. Therefore the condition (i) of Theorem~\ref{thm1} does not occur in this case. Theorem~\ref{thm1} (ii) implies that there exists a large independent set $I(P)$ in $G-E(P)$.
If $m=3$, then instead of using Theorem~\ref{thm1}, we use Theorem~\ref{thm3}. In both cases the following holds.

\begin{lemma}\label{lemaLM}
Let $m\ge 3$, and let $G$ be a cyclically $(2m-1)$-edge-connected cubic graph. 
Let $P$ be a path of length $m$ in $G$. Assume that $G-E(P)$ has no perfect matching.
Then $G-E(P)$ has an independent set containing more than half of the vertices.
\end{lemma}
\noindent In what follows we describe how to recognise this situation using signed graphs.
\smallskip

A \emph{signed graph} $(G,\Sigma)$ is a graph $G$ whose edges are equipped with a positive or a negative sign. The set $\Sigma$ is the set of \emph{negative edges}, while $E(G)-\Sigma$ is the set of \emph{positive} ones. The \emph{switching} at a vertex $v$ of $G$ is inverting the sign of every edge incident with $v$. Two signed graphs $(G,\Sigma_1)$ and $(G,\Sigma_2)$ are \emph{equivalent} if $(G,\Sigma_1)$ can be obtained from $(G,\Sigma_2)$ by a series of switchings. According to Zaslavsky~\cite{zaslavsky}, two signed graphs $(G,\Sigma_1)$ and $(G,\Sigma_2)$ are equivalent if and only if the symmetric difference $\Sigma_1\Delta \Sigma_2$ is an edge-cut of $G$.
\smallskip

Let $I(P)$ be an independent set of vertices in $G-E(P)$ of size greater than $|V(G)|/2$ and among all such sets choose one that has biggest intersection with $P$. Assume that there is $i$, $0<i<m$, such that $p_i$ is not in $I(P)$.
Let $q_i$ be the neighbour of $p_i$ that is not in $\{p_{i-1},p_{i+1}\}$. If $q_i \not \in I(P)$, then we can add $p_i$ to $I(P)$, contradicting the choice of $I(P)$. If $q_i \in I(P)$, then we add $p_i$ to $I(P)$ and remove $q_i$ from $I(P)$. Note that due to cyclic edge-connectivity,
$q_i\neq p_j$ for any $j\in \{0,\ldots,m\}$, which contradicts the choice of $I(P)$. Thus, $p_i$ is in $I(P)$, for each $i\in\{1,\ldots,m-1\}$. 

The set of edges $Y(P)$ is defined as a union of all edges induced by $V(G)-I(P)$ and all edges of $P$ between two vertices of $I(P)$ (thus all inner edges of $P$ are in $Y(P)$). Note that $G-Y(P)$ is bipartite with partites $I(P)$ and $V(G)-I(P)$. The edges of $Y(P) \cap E(P)$ have both end-vertices in $I(P)$, while the edges of $Y(P)-E(P)$ have both end-vertices in $V(G)-I(P)$. Let $(G, Y(P))$ be a signed graph. Switch at the vertices of $I(P)$ and note that the resulting graph has only negative edges. Thus $(G, Y(P))$ is equivalent with $(G, E(G))$. In summary, if the path $P$ cannot be extended to a 2-factor of $G$, then $(G, Y(P))$ is equivalent to an all-negative signed graph.

Let $P$ and $Q$ be two paths of $G$ such that none of them can be extended into a $2$-factor of $G$.
Then $(G, Y(P))$ and $(G, Y(Q))$ are both equivalent to $(G, E(G))$ and thus they are equivalent. 
Therefore, according to Zaslavsky~\cite{zaslavsky}, there is a partition of $V(G)$ into $V_1$ and $V_2$ such that
$E(V_1, V_2)=Y(P) \Delta Y(Q)$, where $E(V_1,V_2)=\{v_1v_2\in E(G)|v_1\in V_1, v_2\in V_2\}$. 

The following lemma thus holds not only if $m \ge 4$ but even for $m=3$ as we can use Theorem~\ref{thm3} instead of Theorem~\ref{thm1}.
\begin{lemma}\label{lemaZas}
Let $m \ge 3$ and let $G$ be a cyclically $(2m-1)$-edge-connected cubic graph.
Let $P$ and $Q$ be two paths of length $m$ in $G$ such that none of them can be extended into a $2$-factor of $G$. Then there is a partition of $V(G)$ into $V_1$ and $V_2$ such that $E(V_1, V_2)=Y(P) \Delta Y(Q)$. 
\end{lemma}

Assuming even higher cyclic edge-connectivity of $G$, only one of the components of $G-E(V_1, V_2)$ contains a cycle. This observation is used in proofs of the remaining theorems. 
We proceed by proving Theorem~\ref{thm4}.

\begin{proof}[Proof of Theorem~\ref{thm4}]
By Theorem~\ref{plesnik}, the result follows for $l\leq 2$. Moreover, if $G\in\{K_2^3,K_4,K_{3,3}\}$, then the theorem holds.
Thus for the rest of the proof, we assume that $l\geq 3$ and that $G\notin\{K_2^3,K_4,K_{3,3}\}$. This allows us to use Lemma~\ref{lemaCyklicka} freely, and apply Lemmas~\ref{lemaLM} and \ref{lemaZas}, because $4l-5 \geq 2l-1$.
 
Assume for the contrary that none of $P_1$, $P_2$, and $P_3$ can be extended into a $2$-factor of $G$. 
Then by Lemma~\ref{lemaLM}, $G-E(P_i)$ has an independent set containing more than half of the vertices, for $i\in\{1, 2, 3\}$.

Let $e\in E(G)$ and let $P$ be a path of $G$. We say that $e$ is \emph{close to $P$} (\emph{distant from $P$}, respectively) if the distance of $e$ from $P$ is at most $4l-9$ (at least $4l-8$, respectively).
Since the distance between any two paths from the theorem statement is at least $8l-16$, we have the following claim.

\medskip

\noindent
\textbf{Claim $1$.} No edge of $G$ is close to both $P_i$ and $P_j$ where $i,j \in \{1,2,3\}$ and $i \neq j$. 

\medskip   

For each path $P_i$, $i \in \{1, 2, 3\}$, we partition the edges of $Y(P_i)$ into three subsets: 
\begin{itemize}
\item $B_i$ are edges of $P_i$, 
\item $C_i$ are edges close to $P_i$ that do not belong to $B_i$, and 
\item $D_i$ are edges distant from $P_i$. 
\end{itemize}

\medskip

Due to the definition of $Y(P_i)$, $B_i$, $C_i$ and $D_i$, the following holds.

\medskip

\noindent
\textbf{Claim 2.} No edge from $B_i$ is adjacent to an edge from $C_i\cup D_i$,  for each $i\in\{1,2,3\}$.
\medskip

The definition of $I(P_i)$ implies the following.
\medskip

\noindent
\textbf{Claim 3.} $|B_i|\in\{l-2,l-1,l\}$, for $i\in\{1,2,3\}$ .
\medskip

According to the last part of Theorem~\ref{thm1} where $d=3$ and $k=l-2$ 
(or according to Theorem~\ref{thm3}
for $l=3$), there are at most $k-1=l-3$ edges of $Y(P_i)$ that are not path edges.
Therefore the following holds.

\medskip

\noindent
\textbf{Claim 4.} $|C_i\cup D_i|\leq l-3$.
\medskip

Without loss of generality, suppose that $P_1$ is a path out of $\{P_1,P_2,P_3\}$ with the minimum number of distant edges. By Claim 1, no edge is close to both $P_2$ and $P_3$. Thus, without loss of generality, we may assume that at most $|D_1|/2$ edges of $D_1$ are close to $P_2$, and denote these edges by $D_1'$. According to Claim~4 and the choice of $P_1$, $|C_2|=|C_2\cup D_2|-|D_2|\leq (l-3)-|D_1|$. Since $2|D'_1|\leq |D_1|$ by the choice of $P_2$, we obtain

\begin{eqnarray}
2\cdot |D'_1| + |C_2| \le l-3. \label{eqd1c2}
\end{eqnarray}

\noindent We are going to show that $|B_2| \leq 2\cdot |D'_1| + |C_2|$, which will be a contradiction with~(\ref{eqd1c2}) by Claim~3. 
\medskip

Since none of $P_1$ and $P_2$ can be extended into a $2$-factor of $G$, signed graphs $(G, Y(P_1))$ and $(G, Y(P_2))$ are both equivalent to the signed graph $(G, E(G))$, and thus they are equivalent to each other. Therefore, according to Lemma~\ref{lemaZas} there is a partition of $V(G)$ into $(V_1, V_2)$ such that $E(V_1, V_2)=Y(P_1) \Delta Y(P_2)$. By Claim 3 and Claim 4, $|E(V_1, V_2)|=|Y(P_1) \Delta Y(P_2)|\leq |Y(P_1)|+|Y(P_2)|\leq |B_1|+|C_1\cup D_1|+|B_2|+|C_2\cup D_2|\leq l+(l-3)+l+(l-3)=4l-6$. 
\medskip

\noindent
\textbf{Claim 5.} $|E(V_1, V_2)|\leq 4l-6$.
\medskip

Due to the small size of $E(V_1, V_2)$ and the condition on the cyclic edge-connectivity of $G$, the components of $G-E(V_1, V_2)$ are mostly acyclic.
\medskip

\noindent
\textbf{Claim 6.} There exists exactly one component of $G-E(V_1, V_2)$ that is cyclic.
 
\begin{proof}[Proof of Claim 6.]
By Claim 5, $|E(V_1, V_2)|\leq 4l-6$. Since $G$ is cyclically $(4l-5)$-edge-connected, at most one of the components of $G-E(V_1, V_2)$, may contain a cycle. Let us show that $G-E(V_1, V_2)$ contains a cyclic component. Assume for the contradiction that $G-E(V_1, V_2)$ is acyclic. But then each part of the cut contains at most $4l-8$ vertices and thus $G$ at most $8l-16$ vertices. This contradicts with the fact that  the distance between $P_1$ and $P_2$ is at least $8l-16$.
\end{proof}

We denote the cyclic component of $G-E(V_1, V_2)$ by $H$. Let $W\subseteq V(G)-V(H)$. Let $E^+(W)$ denote the union of $\partial_G(W)$ with the set of edges induced by $W$ in $G$. If $S$ is an induced subgraph of $G$, then we abbreviate $E^+(V(S))$ to $E^+(S)$.
\medskip

\noindent
\textbf{Claim 7.} Let $W\subseteq V(G)-V(H)$. If $ \partial_G(W)\subseteq E(V_1,V_2)$, then $E^+(W)$ induces an acyclic graph.  

\begin{proof}[Proof of Claim 7.]
By Claim 5, $|\partial_G(H)| \leq 4l-6$. Since $G$ is cyclically $(4l-5)$-edge-connected, any subgraph of $G-V(H)$ is acyclic. Thus the graph induced by $W$ is also acyclic. Suppose that $E^+(W)$ induces a cycle $C$. Then for each edge $e\in\partial_G(W) \cap C$ there is another edge $f\in\partial_G(W) \cap C$ such that $e$ and $f$ are incident with a common vertex $w$ from $V(G)-W$. Note that no cycle of $H$ contains $w$, since either $w\notin H$ or $w$ is a vertex of degree $1$ in $H$. Thus $|\partial_G(W\cup V(C))|\leq |\partial_G(W)|\leq|E(V_1,V_2)|\leq 4l-6$, and $\partial_G(W\cup V(C))$ separates $C$ from the cycle of $H$, which is a contradiction with the cyclic edge-connectivity of~$G$.
\end{proof}

As $H$ is a component of  $G-E(V_1,V_2)$, $H$ is an induced subgraph, which together with Claim 7 implies the following.
\medskip

\noindent
\textbf{Claim 8.} Each edge from $E(G)$ is either in $E(H)$ or in exactly one set $E^+(S)$, where $S$ is an acyclic component of $G-V(H)$. 
\medskip

To bound $|B_2|$ from above, we are going to distinguish whether an edge of $B_2$ belongs to $E(H)$ or to $E^+(S)$ for some acyclic component $S$ of $G-V(H)$. First consider the edges of $H$. For each $e\in B_2\cap E(H)$, $e\notin E(V_1,V_2)$, and thus $e\notin Y(P_1) \Delta Y(P_2)$. Therefore, as $e\in Y(P_2)$,
we have $e\in Y(P_1)$. According to Claim 1 and the definition of $D'_1$, $e\in D'_1$, and
\begin{eqnarray}
|B_2 \cap E(H)| \leq |D'_1 \cap E(H)|. \label{eqinH}
\end{eqnarray}

Assume now that $S$ is an (acyclic) component of $G-V(H)$. Due to the definition of $H$, all vertices of $H$ belong either to $V_1$ or to $V_2$.
Note that the edges of $\partial_G(S)$ have one end in $S$ and one end in $H$.  Hence $\partial_G(S)\subseteq E(V_1,V_2)=Y(P_1) \Delta Y(P_2)$
and all vertices of $S$ incident with an edge from $\partial_G(S)$ belong either to $V_1$ or to $V_2$.
The following lemma bounds the number of edges in $|B_2 \cap E^+(S)|$.

\begin{lemma}\label{lemmathm3}
Let $T$ be an induced subtree of $G-V(H)$ such that $ \partial_G(T)\subseteq Y(P_1) \Delta Y(P_2)$. Assume further that all vertices of $T$ incident with edges from $\partial_G(T)$ belong either to $V_1$ or to $V_2$.
Then $|B_2 \cap E^+(T)| \leq 2 |D'_1 \cap E^+(T)| + |C_2 \cap E^+(T)|$. 
\end{lemma}

\begin{proof}[Proof of Lemma~\ref{lemmathm3}.] The proof is by complete induction on $|E(T) \cap (Y(P_1) \Delta Y(P_2))|$. 

If $|B_2 \cap E^+(T)| = 0$, then the lemma holds. Assume for the rest of the proof that $|B_2 \cap E^+(T)| \geq 1$. 

By the assumption of the lemma and by Claim 5,
$|\partial_G(T)| \le |Y(P_1) \Delta Y(P_2)|=|E(V_1,V_2)| \le 4l-6$.
Thus $T$ has at most $4l-8$ vertices and the edges of $E^+(T)$ are close to $P_2$. By Claim 1, they are distant from $P_1$. 
Thus if an edge from  $E^+(T)$ belongs to $Y(P_1)$, then it must be from $D_1'$. On the other hand, if an edge from  $E^+(T)$ belongs to $Y(P_2)$, then it is either from $B_2$ or from $C_2$.
\smallskip

If $T$ is an isolated vertex, then either $|\partial_G(T)\cap B_2|=1$ or $|\partial_G(T)\cap B_2|=2$. Let $e\in \partial_G(T)- B_2$.  By Claim 2, $e\notin C_2\cup D_2$. Thus, $e\in Y(P_1)$, and hence $e\in D'_1$. Therefore,

\begin{eqnarray}
|B_2 \cap E^+(T)| \leq 2|D'_1 \cap E^+(T)| \nonumber
\end{eqnarray}

\noindent
and the lemma holds.

Assume for the rest of the proof that $T$ is not an isolated vertex. If $|\partial_G(T)|\geq 2l$, then since $|B_2|\leq l$, the following holds.

\begin{eqnarray}
|B_2 \cap \partial_G(T)| \leq |D'_1 \cap \partial_G(T)| + |C_2 \cap \partial_G(T)|. \label{eqinHpart}
\end{eqnarray}
If, on the other hand, $|\partial_G(T)|< 2l$, then $E(P_2) \cap E^+(T)$ induces a connected graph in $G$. 
Indeed, if $P_2 \cap E^+(T)$ induces a graph with two components there exists a subpath $Q$ of $P_2$ of length at most $l$
that connects two vertices of $T$ and for which $E(Q)\cap E(T)=\emptyset$.
As $T$ is connected, endvertices of $Q$ can be connected by a path $Q'$ in $T$. Since $|\partial_G(T)|< 2l$ and $T$ is a tree, $|V(T)|<2l-2$ and $Q'$ has at most $2l-4$ edges. 
The paths $Q$ and $Q'$ form a cycle of total length at most $3l-4$.
As $G$ is cyclically $(4l-5)$-edge-connected, this contradicts Lemma~\ref{lemaCyklicka}. 
Thus $E(P_2) \cap E^+(T)$ induces a connected graph in $G$.
By Claim 7, $E^+(T)$ is acyclic, hence $|B_2 \cap \partial_G(T)| \le 2$. But $|\partial_G(T)-B_2| \ge 2$ as $T$ is not an isolated vertex. Thus the inequality~(\ref{eqinHpart}) holds.
\medskip

Let $T_0$ be a union of such components of $T-(Y(P_1) \Delta Y(P_2))$ that contain a vertex of $T$ that is incident with an edge from $\partial_G(T)$. Recall that the vertices of $T$ incident with an edge from $\partial_G(T)$ all belong either to $V_1$ or to $V_2$, say $V_1$. Thus the same holds for the vertices of $T_0$ that are incident with $\partial_G(T)$. Therefore all vertices of $T_0$ must also belong to $V_1$, and so $T_0$ is an induced subgraph of $G$. Hence if $e\in E(T_0)\cap B_2$, then $e\in Y(P_1)$ and so $e\in D_1'$.  Thus
\begin{eqnarray}
|B_2 \cap E(T_0)| \leq |D'_1 \cap E(T_0)|. \label{eqinHpart2}
\end{eqnarray}

Consider a component $T'$ of $T-V(T_0)$. Let us show that $T'$ satisfies the conditions of this lemma. It is easy to see that $T'$ is an induced subtree of $G-V(H)$. By the definition of $T_0$, $\partial_G(T')\subseteq Y(P_1) \Delta Y(P_2)$ and all edges of $\partial_G(T')$ must have one end-vertex in $T_0$ and the other end-vertex in $T'$. Since the first one belongs to $V_1$, the second one must be in $V_2$. Since $T'$ does not contain $\partial_G(T')$, $|E(T') \cap (Y(P_1) \Delta Y(P_2))|< |E(T) \cap (Y(P_1) \Delta Y(P_2))|$. Applying induction to $T'$,  
\begin{eqnarray}
|B_2 \cap E^+(T')| \leq 2 |D'_1 \cap E^+(T')| + |C_2 \cap E^+(T')| \label{eqinHpart3}.
\end{eqnarray}
The inequalities (\ref{eqinHpart}), (\ref{eqinHpart2}), together with
the inequality  (\ref{eqinHpart3}) applied to each component $T'$ of $T-V(T_0)$ imply the statement of the lemma.
\end{proof}

By Claim~8, it follows from (\ref{eqinH}) and from Lemma~\ref{lemmathm3} that $|B_2|\le 2|D'_1|+|C_2|$. This is a contradiction with (\ref{eqd1c2}) by Claim 3, which concludes the proof of the theorem. 
\end{proof}


\section{On $3$-paths and $4$-paths}\label{s4}
\begin{lemma}\label{lemaFarb}
Let $G$ be a graph, let $A, B\subseteq E(G)$ such that $G-A$ and $G-B$ are bipartite with partites $V^A_1, V^A_2$ and $V^B_1,V^B_2$, respectively. Moreover, assume that the endpoints of each edge of $A$ belong to the same partite of $G-A$ and the endpoints of each edge of $B$ belong to the same partite of $G-B$.
Then there exists a partition of $V(G)$ into $V_1$ and $V_2$ (one of the sets may be empty) such that 
\begin{enumerate}
\item $E(V_1,V_2)=A\Delta B$, and 
\item $V^B_1=V^A_1\Delta V_1 = V^A_2\Delta V_2$ and $V^B_2=V^A_2\Delta V_1 = V^A_1\Delta V_2$.
\end{enumerate} 
\end{lemma}

\begin{proof}
Let us consider the signed graphs $(G,A)$ and $(G,B)$. Note that switching at vertices of $V^A_1$ in $(G,A)$ results in an all-negative signed graph $(G,E(G))$. Similarly, $(G,B)$ can be switched to $(G,E(G))$, and hence $(G,A)$ and $(G,B)$ are equivalent. 
Thus we can obtain $(G,B)$ from $(G,A)$ by switching at the vertices
$V^A_1\Delta V^B_1$. Let $V_1=V^A_1\Delta V^B_1$ and $V_2=V(G)-V_1$ ($=V^A_2\Delta V^B_1$).  
As $V^A_1\Delta V^B_1=V^A_2\Delta V^B_2$ and $V^A_2\Delta V^B_1=V^A_1\Delta V^B_2$ the statement 2 of this lemma holds.

Note that $E(V_1,V_2)=\partial_G(V_1)=\partial_G(V^A_1\Delta V^B_1)=
\partial_G(V^A_1) \Delta \partial(V^B_1)$.
Since switching at all vertices of a set $W$ changes signs exactly of the edges in $\partial(W)$, from the fact that switching at $V^A_1$ in $(G, A)$ produces $(G,E(G))$ 
and switching at $V^B_1$ in $(G, B)$ produces $(G,E(G))$ 
we derive
$A \Delta \partial_G(V^A_1)=E(G)=B \Delta \partial_G(V^B_1)$
and thus
$\partial_G(V^A_1) \Delta \partial_G(V^B_1)=A \Delta B$
which concludes the proof of the statement 1 of this lemma.
\end{proof}

\begin{figure}[p]
\begin{center} 
\includegraphics[scale=1.2]{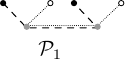} \\
\ \\
\includegraphics[scale=1.2]{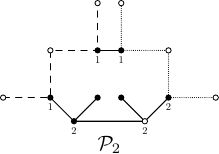} \ \ \ \ 
\includegraphics[scale=1.2]{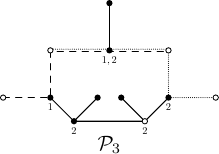} \\
\ \\
\includegraphics[scale=1.2]{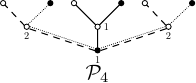} \ \ \ \ 
\includegraphics[scale=1.2]{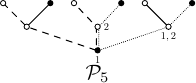} \\
\ \\
\includegraphics[scale=1.2]{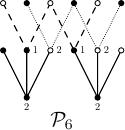} \ \ \ \ 
\includegraphics[scale=1.2]{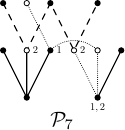} \ \ \ \ 
\includegraphics[scale=1]{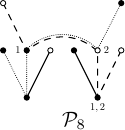} \\
\ \\
\includegraphics[scale=1.2]{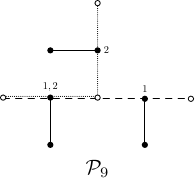} \ \ \ \
\includegraphics[scale=1.2]{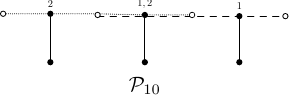} \\
\ \\
\includegraphics[scale=1.2]{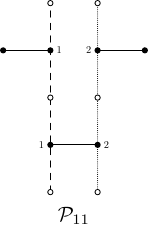} \ \ \ \ 
\includegraphics[scale=1.2]{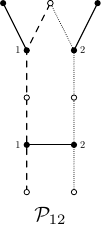} \ \ \ \ 
\includegraphics[scale=1.2]{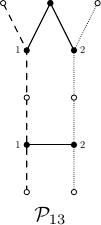} 
\caption{Positions $\mathcal{P}_1$, $\mathcal{P}_2$, \ldots, $\mathcal{P}_{13}$ of paths related to Theorems~\ref{thm5} and~\ref{thm6}. The respective paths are depicted by dashed and dotted lines.}
\label{fig34}
\end{center}
\end{figure}

Let $G$ be a graph and let $P_1=v_1v_2v_3v_4$ and $P_2=w_1w_2w_3w_4$ be two paths of $G$ of length $3$. We say that $P_1$ and $P_2$ are in the \textit{position $\mathcal{P}_1$} if $v_2v_3=w_2w_3$, $|\{v_1,v_4,w_1, w_4\}|=4$, and $G-\{v_2,v_3\}$ admits a proper (vertex) $2$-colouring $c$ such that $c(v_1)=c(v_4)\neq c(w_1)=c(w_4)$. This position is depicted on the top of Figure~\ref{fig34}.

\begin{theorem} \label{thm5}
Let $G$ be a cyclically $5$-edge-connected cubic graph and let $P_1$ and $P_2$ be
two distinct paths of $G$ of length $3$. Then there is a $2$-factor of $G$ that contains $P_1$ or $P_2$ if and only if $P_1$ and $P_2$ are not in the position $\mathcal{P}_1$.
\end{theorem}

\begin{proof}
If $G\in\{K_2^3,K_4,K_{3,3}\}$, then the theorem holds. Thus we assume that $G\notin\{K_2^3,K_4,K_{3,3}\}$, and we can use Lemma~\ref{lemaCyklicka} without considering these exceptions.
Let $i\in\{1,2\}$. Suppose that $P_i$ cannot be extended into a $2$-factor of $G$. Equivalently, $G-E(P_i)$ has no perfect matching. By Theorem~\ref{thm3}, $G-E(P_i)$ is bipartite and edges from $E(P_i)$ are incident with vertices from the same partite. 
Let $p_i$ be an edge of $P_i$ incident with an end-vertex of $P_i$, and let $q_i$ be an edge of $G-E(P_i)$ that is incident with the central vertex of $P_i-p_i$. Note that $G-\{p_i,q_i\}$ is bipartite. Moreover, there are two choices how to choose $p_1$ and $q_1$ thus we may assume that $\{p_1, q_1\} \neq \{p_2, q_2\}$.

By the first statement of Lemma~\ref{lemaFarb}, $\{p_1,q_1\}\Delta \{p_2,q_2\}$ is an edge-cut of $G$. Since $G$ is cyclically $5$-edge-connected and cubic, $|\{p_1,q_1\}\Delta \{p_2,q_2\}|\neq 2$. Thus, $\{p_1,q_1\}\Delta \{p_2,q_2\}$ is a $4$-edge-cut, and, again by cyclic edge-connectivity, one side of the edge-cut $\{p_1,q_1\}\Delta \{p_2,q_2\}$ induces an edge, which we denote by $e$. 

The edges $p_i$ and $q_i$ are not adjacent, because $G$ does not contain any triangle according to Lemma~\ref{lemaCyklicka}. Again by Lemma~\ref{lemaCyklicka}, $G$ has no $4$-circuit and therefore, $e$ is the only edge that is neighbour of $p_i$ and $q_i$ both. Thus $e\in P_i$. 

If $p_1$ is adjacent to $q_2$, then $p_2$ is adjacent to $q_1$. It follows that $P_1=P_2$, which contradicts that $P_1$ and $P_2$ are distinct.
Thus $p_1$ is adjacent to $p_2$, $q_1$ is adjacent to $q_2$,
 and $E(P_i)=\{p_i,e,q_{3-i}\}$, for $i\in\{1,2\}$.  Let $P_1=v_1v_2v_3v_4$, $P_2=w_1w_2w_3w_3$ with the notation chosen in such a way that $p_1=v_1v_2$ and $p_2=w_1w_2$ (and thus $q_1=w_3w_4$ and $q_2=v_3v_4$). 

Denote the partites of $G-\{p_1,q_1\}$ by $V_1^{P_1}$ and $V_2^{P_1}$ in such a way that $v_1 \in V_1^{P_1}$.
Denote the partites of $G-\{p_2,q_2\}$ by $V_1^{P_2}$ and $V_2^{P_2}$ in such a way that $v_1 \in V_1^{P_2}$. Let $V_1$ and $V_2$ be a partition of $V(G)$ guaranteed by Lemma~\ref{lemaFarb}. Thus $E(V_1,V_2)=\{p_1,q_1,p_2,q_2\}$ is an edge-cut of $G$.
Since $v_1 \in V_1^{P_1} \cap V_1^{P_2}$, the second statement of Lemma~\ref{lemaFarb}
implies that $v_1 \in V_2$. Therefore $V_1=\{v_2,v_3\}$, 
$V^{P_2}_1=V^{P_1}_1\Delta \{v_2,v_3\}$ and
$V^{P_2}_2=V^{P_1}_2\Delta \{v_2,v_3\}$.

Thus
$v_1 \in V_1^{P_1} \cap V_1^{P_2}$ implies that
$v_2 \in V_1^{P_1} \cap V_2^{P_2}$,
$w_1 \in V_2^{P_1} \cap V_2^{P_2}$,
$v_3 \in V_2^{P_1} \cap V_1^{P_2}$,
$v_4 \in V_1^{P_1} \cap V_1^{P_2}$, and
$w_4 \in V_2^{P_1} \cap V_2^{P_2}$.
We define a proper vertex $2$-colouring of $G-\{v_2,v_3\}$.
Colour the vertices of $V_1^{P_1}-\{v_2\}$ with white and vertices of $V_2^{P_1}-\{v_3\}$ with black.  For this colouring $c$ we have $c(v_1)=c(v_4)\neq c(w_1)=c(w_4)$.
Thus $P_1$ and $P_2$ are in the position $\mathcal{P}_1$.

For the converse, suppose that $P_1$ and $P_2$ are in the described position, and let $i\in\{1,2\}$. Then the vertices of the colour used for end-vertices of $P_i$ together with the end-vertices of $e$ form a large independent set satisfying Theorem~\ref{thm3} (b) for $P_i$. Thus, there is no perfect matching avoiding $P_i$, and consequently, $P_i$ does not belong to any $2$-factor. 
\end{proof}

Let $G$ be a graph and let $P_1$ and $P_2$ be two paths of $G$ of length $4$. We say that $P_1$ and $P_2$ are in the \textit{position $\mathcal{P}_j$}, for $j\in\{2,3,4,5,6,7,8\}$, if their positions correspond to the respective part of Figure~\ref{fig34} and $G$ admits a (vertex) $2$-colouring that is proper except of three monochromatic edges depicted on the figure. The paths are depicted by dashed and dotted lines and all the displayed vertices are distinct.

\begin{theorem} \label{thm6}
Let $G$ be a cyclically $7$-edge-connected cubic graph and let $P_1$ and $P_2$ be two paths of $G$ of length $4$. Then there is a $2$-factor of $G$ that contains $P_1$ or $P_2$ if and only if
\begin{itemize}
\item there is a $2$-factor containing all the edges in $E(P_1) \cap E(P_2)$ (this is always true if $|E(P_1) \cap E(P_2)| \le 2$), and
\item for each two subpaths $P'_1$ of $P_1$ and $P'_2$ of $P_2$, both on $3$ edges,  $P'_1$ and $P'_2$ are not in the position $\mathcal{P}_1$, and 
\item the paths $P_1$ and $P_2$ are not in position $\mathcal{P}_i$ for each $i\in \{2, 3, \dots, 13\}$.
\end{itemize}
\end{theorem}

\begin{proof}
If $G\in\{K_2^3,K_4,K_{3,3}\}$, then the theorem holds. Thus we assume that $G\notin\{K_2^3,K_4,K_{3,3}\}$, and we can use Lemma~\ref{lemaCyklicka} without considering these exceptions.
Let $i\in\{1,2\}$. Suppose first that $P_i$ cannot be extended into a $2$-factor of $G$. Equivalently, $G-E(P_i)$ has no perfect matching. By Lemma~\ref{lemaLM}, $G-E(P_i)$ has an independent set containing more than half of the vertices. Let $I(P_i)$ be defined in the same way as in Section~3; it is an independent set in $G-E(P_i)$ that contains more than half of the vertices and that has maximum intersection with $V(P_i)$ among all such sets. Due to this definition all inner vertices of $P_i$ belong to $I(P_i)$. Let us show that at least one end-vertex of $P_i$ also belongs to $I(P_i)$. 

Let $J(P_i):=V(G)-I(P_i)$, and let $s$ ($t$, respectively) be the number of edges in $G$ induced by $I(P_i)$ ($J(P_i)$, respectively). Note that $s\leq 4$, because $I(P_i)$ is an independent set in $G-E(P_i)$ and $|E(P_i)|=4$. Then $\partial_G(I(P_i))=3|I(P_i)|-2s=3|J(P_i)|-2t=\partial_G(J(P_i))$. If $|I(P_i)|\geq |J(P_i)|+4$, then $s\geq 6+t$, which is impossible since $s\leq 4$. Since $|I(P_i)|>|J(P_i)|$, $|I(P_i)|= |J(P_i)|+2$ and $s=t+3$. Hence either $s=3$ or $s=4$, and thus either $|V(P_i)\cap I(P_i)|=4$ or $|V(P_i)\cap I(P_i)|=5$, respectively. In the former case, $J(P_i)$ is an independent set, and $G-E(P_i)$ is bipartite. In the latter case, $J(P_i)$ induces a graph with one edge, denoted by $r_i$. 

Consider now the case when $|V(P_i)\cap I(P_i)|=5$.
Note that the end-vertices of $r_i$ belong to a different partite of $G-(E(P_i)\cup\{r_i\})$ than the vertices of $P_i$.

\smallskip
\noindent
\textbf{Claim 1.} If $|V(P_i)\cap I(P_i)|=5$, then $r_i$ is not adjacent to any edge of $E(P_i)$, for $i\in\{1,2\}$.
\medskip

Let $P_1=v_1v_2v_3v_4v_5$, and let $v_2v'_2$ and $v_4v'_4$ be edges of $G-E(P_1)$. If $|V(P_1)\cap I(P_1)|=4$, then assume, without loss of generality, that $V(P_1)\cap I(P_1)=\{v_1,v_2,v_3,v_4\}$, and let $X_1=\{v_2v'_2,v_3v_4\}$.  
Note that $G-X_1$ is bipartite with one of the partite sets being $J(P_1)\cup\{v_2\}$. If $|V(P_1)\cap I(P_1)|=5$, then 
we define $X_1=\{v_2v'_2,v_4v'_4,r_1\}$.
The graph $G-X_1$ is bipartite with one of the partite sets being $J(P_1)\cup\{v_2,v_4\}$. Moreover, all end-vertices of the edges from $X_1$ belong to the same partite of $G-X_1$. 
For $P_2$ we define $X_2$ analogously and we obtain the following claims.

\smallskip
\noindent
\textbf{Claim 2.} 
If $|V(P_i)\cap I(P_i)|=4$, then the graph $G-X_i$ is bipartite, for $i\in\{1,2\}$. 
\smallskip

\smallskip
\noindent
\textbf{Claim 3.} 
If $|V(P_i)\cap I(P_i)|=5$, then the graph $G-X_i$ is bipartite. All end-vertices of the edges from $X_i$ connect vertices from the same partite of $G-X_i$, for $i\in\{1,2\}$. 
\smallskip

We consider four cases.
\smallskip

\textbf{Case 1.} $|V(P_1)\cap I(P_1)|=4$ and $|V(P_2)\cap I(P_2)|=4$.
\\ If paths $v_1v_2v_3v_4$ and $w_1w_2w_3w_4$ are distinct then they fulfil assumptions of Theorem~\ref{thm5}. It follows that the paths $v_1v_2v_3v_4$ and $w_1w_2w_3w_4$ are in the position $\mathcal{P}_1$. If they are equal, then either
$v_1=w_1$ or $v_1=w_4$. If $v_1=w_1$ and $v_5=w_5$, then the paths are identical and the desired $2$-factor exists only if there exist a $2$-factor containing all the edges in the intersection. 
If $v_1=w_1$ and $v_5\neq w_5$, then if there exists a $2$-factor containing the edges in the intersection, it has to continue to either $v_5$ or $w_5$.
Thus no such $2$-factor exists, a contradiction.

On the other hand if, $v_1=w_4$ (and thus $v_2=w_3$, $v_3=w_2$, and $v_4=w_1$). Due to the choice of $v_5$ the set $I(P_1)$ contains $v_4$ and does not contain $v_5$.
But then $I(P_1)$ is an independent set in $G-\{v_1v_2, v_2v_3, v_3v_4\}$ and it contains more than half of the vertices of $G$. Thus there is no $2$-factor containing $E(P_1) \cap E(P_2)$, a contradiction.
\smallskip

\textbf{Case 2.} $|V(P_1)\cap I(P_1)|=4$ and $|V(P_2)\cap I(P_2)|=5$.
\\By Claims 2 and 3 and by Lemma~\ref{lemaFarb}, $X_1\Delta X_2$ is an edge-cut of $G$. Since $G$ is cyclically $7$-edge-connected and cubic, $|X_1\Delta X_2|\neq 1$. Thus, either $|X_1 \Delta X_2|=3$ or $|X_1\Delta X_2|=5$. 

Assume first that $|X_1\Delta X_2|=3$. Due to cyclic edge-connectivity of $G$, one side of $G-(X_1\Delta X_2)$ must be a vertex $a$. Up to symmetry of $w_2w'_2$ and $w_4w'_4$, there are four options: $v_3v_4=w_2w'_2$, $v_2v'_2=w_2w'_2$, $v_3v_4=r_2$, $v_2v'_2=r_2$. In the first two cases, by Claim~1, $a=w'_4$. By Lemma~\ref{lemaCyklicka}, none of $\{w_1,w_2,w'_2,w_3,w_5\}$ is incident with an edge from $X_1\Delta X_2$. Since one of the edges from $\{v_3v_4,v_2v'_2\}$ is an edge from $X_1\Delta X_2$ while the other one equals $w_2w'_2$, the set of edges $\{v_2v'_2,v_2v_3,v_3v_4,w_2w_3,w_3w_4,w_4w'_4\}$ induces a short cycle, contradicting Lemma~\ref{lemaCyklicka}.
In the latter two cases, since $w_2w'_2$ are $w_4w'_4$ are both incident with $a$, there is a short cycle as well, a contradiction with Lemma~\ref{lemaCyklicka}. 

Assume now that $|X_1\Delta X_2|=5$. Due to cyclic edge-connectivity of $G$, one side of $G-(X_1\Delta X_2)$ must induce a path $abc$ on three vertices $a,b,c$. Note that $w_2w'_2$ and $w_4w'_4$ cannot be incident either with the same vertex or with two adjacent vertices, as otherwise $G$ would contain a $4$-circuit or a $5$-circuit, respectively, a contradiction with Lemma~\ref{lemaCyklicka}. Hence, without loss of generality, we may assume that $w_2w'_2$ is incident with $a$ and $w_4w'_4$ is incident with $c$. By Claim 1, $w_2\neq a$ and $w_4\neq c$. We conclude that the path $\{w_2, w_3, w_4, a, b, c\}$ induce either a $4$-circuit or a $6$-circuit, a contradiction.
\smallskip

\textbf{Case 3.} $|V(P_1)\cap I(P_1)|=5$ and $|V(P_2)\cap I(P_2)|=4$.
\\This case is equivalent to Case~2. 
\smallskip

\textbf{Case 4.} $|V(P_1)\cap I(P_1)|=5$ and $|V(P_2)\cap I(P_2)|=5$.
\\By Claim~3 and by Lemma~\ref{lemaFarb}, $X_1\Delta X_2$ is an edge-cut. Since $G$ is cyclically $7$-edge-connected and cubic, $|X
_1\Delta X_2|\neq 2$. Thus, either 
$X_1=X_2$, or $|X_1\Delta X_2|=4$ or $|X_1\Delta X_2|=6$.
\smallskip

\textbf{Case 4a.} $X_1 = X_2$.
\\We may without loss of generality assume that $v_2v'_2=w_2w'_2$.
Due to Claim~3, $G-X_1$ is bipartite and edges from $X_1$ are incident with vertices from the same partite.
Assume first $v_2=w_2$ and in addition $v_1=w_1$ and thus $v_3=w_3$. 
If $v_4=w_4$ then also $v_5=w_5$ as Claim 1 implies $v_4v'_4=w_4w'_4$.
Thus assume $v_4\neq w_4$. Due to Lemma~\ref{lemaCyklicka} we have $v_4v'_4=r_2$ and
$w_4w'_4=r_1$ which shows that $P_1$ and $P_2$ are in the position $\mathcal{P}_9$.  Now assume that $v_2=w_2$ and in addition $v_1 = w_3$ and thus $v_3=w_1$. 
Due to Lemma~\ref{lemaCyklicka} we have $v_4v'_4=r_2$ and
$w_4w'_4=r_1$ which shows that $P_1$ and $P_2$ are in the position $\mathcal{P}_{10}$. 

Finally, if $v_2 \neq w_2$ then due to Lemma~\ref{lemaCyklicka} and Claim~1 
$v_1$, $w_1$, $v_2$, $w_2$, $v_3$, $w_3$, $v_4$, $w_4$ are all distinct and $v_4v'_4=r_2$ and
$w_4w'_4=r_1$. It is, however, possible that either $v_5=w_5$ or $v'_4=w'_4$.
This shows that $P_1$ and $P_2$ are in one of the positions $\mathcal{P}_{11}$, $\mathcal{P}_{12}$, $\mathcal{P}_{13}$. 
\smallskip

\textbf{Case 4b.} $|X_1\Delta X_2|=4$. 
\\Due to cyclic edge-connectivity of $G$, one side of $G-(X_1\Delta X_2)$ contains a single edge $ab$. If $v_2v'_2$ ($w_2w'_2$, respectively) is incident with $a$ or $b$, then $v_4v'_4$ ($w_4w'_4$, respectively) is incident neither with $a$ nor with $b$ and vice versa, since otherwise $G$ contains a cycle of length at most 5, contradicting Lemma~\ref{lemaCyklicka}. Hence, we may assume, without loss of generality, that $X_1\Delta X_2=\{v_2v'_2, w_2w'_2,r_1,r_2\}$. Thus $v_4v'_4=w_4w'_4$. 
Let $V_1^{X_1}$ and $V_2^{X_1}$ ($V_1^{X_2}$ and $V_2^{X_2}$, respectively) be partites of $G-X_1$ ($G-X_2$, respectively) such that $a\in V_1^{X_1}$ ($a\in V_2^{X_2}$, respectively). Let $V_1$ and $V_2$ be a partition of $V(G)$ guaranteed by Lemma~\ref{lemaFarb}. Note that $E(V_1,V_2)=X_1\Delta X_2$ and $V_1=\{a,b\}$. Thus $b\in V_2^{X_1} \cup V_1^{X_2}$. Since $a$ and $b$ belong to different partites of $G-X_1$, by Claim~3, it follows that $v_2v'_2$ and $r_1$ are adjacent to the same vertex, say $a$. Then $w_2w'_2$ and $r_2$ are incident with $b$. By Claim~1, $a=v'_2$ and $b=w'_2$. Let $r_1:=aa'$ and $r_2=bb'$. 
Let $U:=\{v_1,v_2,v_3,w_1,w_2,w_3, a,b,a',b'\}$. 
Note that the vertices of $U$ must be all distinct by Lemma~\ref{lemaCyklicka}. Moreover, again by Lemma~\ref{lemaCyklicka}, $v_4\notin U$ and $w_4\notin U$. However, it is possible that $v_4=w_4$.

Assume first that $v_4\neq w_4$. Since $v_4v'_4=w_4w'_4$, $v_4=w'_4$ and $w_4=v'_4$. By definition of $v'_4$ and $w'_4$ and by Lemma~\ref{lemaCyklicka}, $v_5, w_5\notin U\cup\{v_4,w_4\}$ and $v_5\neq w_5$. Colour the vertices according to the partitions $V_1^{X_1}$ and $V_2^{X_1}$. Lemma~\ref{lemaFarb} allows us to determine the colouring of the vertices from $U\cup \{v_4,v'_4,v_5,w_5\}$. This shows that $P_1$ and $P_2$ are in the position $\mathcal{P}_2$. 

Assume now that $v_4=w_4$. Thus $v_4$ is incident with $v_3$ and $w_3$. Since $v_4v'_4=w_4w'_4$ and by definition of $v'_4$ and $w'_4$, $v_4v'_4\notin P_1\cup P_2$ and hence $v'_4=w'_4$ is the third vertex incident with $v_4$. Therefore, $v_5=w_3$ and $w_5=v_3$. Colour the vertices according to the partitions $V_1^{X_1}$ and $V_2^{X_1}$. Lemma~\ref{lemaFarb} allows us to determine the colouring of the vertices from $U\cup \{v_4,v'_4\}$. This shows that $P_1$ and $P_2$ are in the position $\mathcal{P}_3$. 
\smallskip

\textbf{Case 4c.} $|X_1\Delta X_2|=6$.
\\It remains to consider that $|X_1\Delta X_2|=6$.  Due to cyclic edge-connectivity of $G$, one side of $G-(X_1\Delta X_2)$ is either a path on $4$ vertices or a star with three pendant vertices or two isolated vertices.

Let one side of $G-(X_1\Delta X_2)$ be a path on $4$ vertices $a$, $b$, $c$ and $d$. By Lemma~\ref{lemaCyklicka}, $v_2v'_2$ and $v_4v'_4$ are not adjacent. By Claim~3, $v_2v'_2$, $v_4v'_4$ and $r_1$ must be incident with $a$ and $c$, or with $b$ and $d$, say, without loss of generality, $v_2v'_2$ and $r_1$ are incident with $a$ and $v_4v'_4$ is incident with $c$. By Claim~1, $a\neq v_2$. This is a contradiction with Lemma~\ref{lemaCyklicka}, since $\{v_2,v_3,v_4,a,b,c\}$ induces a short cycle. 

Let one side of $G-(X_1\Delta X_2)$ be a star with a central vertex $a$ and three pendant vertices $b,c$ and $d$. Note that by Lemma~\ref{lemaCyklicka}, $v_2v'_2$ and $v_4v'_4$ as well as $w_2w'_2$ and $w_4w'_4$ cannot be adjacent. Let, without loss of generality, $v_2v'_2$ and $w_2w'_2$ be incident with $b$, let $v_4v'_4$ be incident with $c$. Note that $a=v_3=w_3$ by Lemma~\ref{lemaCyklicka}, because $v_2v'_2$ and $v_4v'_4$ ($w_2w'_2$ and $w_4w'_4$, respectively) belong to a path of length 4. Thus $b=v_2=w_2$, $c=v_4$, $v_1=w'_2$ and $w_1=v'_2$. 

Let us first assume that $w_4w'_4$ is incident with $c$. Thus $c=w_4$,  $v_5=w'_4$ and $w_5=v'_4$. Moreover, two neighbours of $d$ different from $a$ are distinct from vertices $v_1,v_2,v_3,v_4,v_5,w_1,w_5$. Since $G-X_1$ is bipartite, it follows that $P_1$ and $P_2$ are in the position~$\mathcal{P}_4$.

Let us now assume that $w_4w'_4$ is not incident with $c$, then $d=w_4$, $r_2=v_4v_5$ and $r_1=w_4w_5$. Since $G-X_1$ is bipartite, it follows that $P_1$ and $P_2$ are in the position~$\mathcal{P}_5$.

Finally, let one side of $G-(X_1\Delta X_2)$ be two isolated vertices $a$ and $b$. By Lemma~\ref{lemaCyklicka}, $v_2v'_2$ and $v_4v'_4$ as well as $w_2w'_2$ and $w_4w'_4$ cannot be adjacent. Thus, without loss of generality, we may assume that $v_2v'_2$, $w_2w'_2$ and $r_1$ are incident with $a$ while $v_4v'_4$, $w_4w'_4$ and $r_2$ are incident with $b$. By Claim~1, $a\neq v_2$ and $b\neq w_4$. Let $r_1=aa'$ and $r_2=bb'$. We distinguish cases according to positions $v_4$ and $w_2$.

First, assume that $w_2\neq a$ and $v_4\neq b$. Then $v_3$, which is a neighbour of $v_2$ and $v_4$, is different from $a=v'_2$ and $b=v'_4$. Similarly, $w_3$, which is a neighbour of $w_2$ and $w_4$, is different from $a=w'_2$ and $b=w'_4$. Moreover, $v_1$, $w_1$, $v_5$ and $w_5$ is the third vertex adjacent to $v_2$, $w_2$, $v_4$ and $w_4$, respectively. Due to Lemma~\ref{lemaCyklicka}, the vertices $v_1,v_2,v_3,v_4,v_5,w_1,w_2,w_3,w_4,w_5, a,a',b,b'$ are all distinct. Let $V_1^{X_1}$ and $V_2^{X_1}$ be the partites of $G-X_1$ such that $a\in V_1^{X_1}$. By Claim~3, this is the position $\mathcal{P}_6$.

Second, assume that $w_2\neq a$ and $v_4=b$. Then $w_3$, which is a neighbour of $w_2$ and $w_4$, is different from $a=w'_2$ and $b=w'_4$. Note that $b=v_4=w'_4$ is adjacent to $v'_4,w_4$ and $b'$. The vertex, $v_3$, which is a neighbour of $v_2$ and $v_4$, is different from $v'_4$ by the definition of $v'_4$. Moreover, $v_3$ is different from $w_4$, as otherwise $G$ contains a $5$-cycle $v_2v'_2w_2w_3w_4v_2$, contradicting Lemma~\ref{lemaCyklicka}. Thus, $v_3=b'$ and therefore $v_5=w_4$. Note that by Lemma~\ref{lemaCyklicka}, $v_1,v_2,v'_4,w_1,w_2,w_3,w_4,w_5,a,a',b,b'$ are distinct vertices. Let $V_1^{X_1}$ and $V_2^{X_1}$ be the partites of $G-X_1$ such that $a\in V_1^{X_1}$. By Claim~3, this is the position $\mathcal{P}_7$. 

The case when $w_2=a$ and $v_4\neq b$ is similar to the previous one. This yields the position $\mathcal{P}_7$ of the paths.

Finally assume that $w_2=a$ and $v_4=b$. Since $v_3$ is a neighbour of $v_2$ and $v_4$, and by Lemma~\ref{lemaCyklicka}, either $v_2w_4\in E(G)$ or $v_2b'\in E(G)$. Similarly, since $w_3$ is a neighbour of $w_2$ and $w_4$, and by Lemma~\ref{lemaCyklicka}, either $v_2w_4\in E(G)$ or $w_4a'\in E(G)$. Therefore, either $v_2w_4\in E(G)$ or $\{v_2b',w_4a'\}\subseteq E(G)$. Since $v_2aa'w_4bb'v_2$ would be a $6$-cycle, contradicting Lemma~\ref{lemaCyklicka}, $v_2w_4$ must be an edge of $G$. Thus $w_3=v_2$ and $v_3=w_4$. Consequently $a'=w_1$ and $b'=v_5$. By Lemma~\ref{lemaCyklicka}, $v_1,v_2,v_3,v_4,v_5,v'_4,w_1,w_2,w_5,w'_2$ are all distinct. Let $V_1^{X_1}$ and $V_2^{X_1}$ be the partites of $G-X_1$ such that $a\in V_1^{X_1}$. By Claim~3, this is the position $\mathcal{P}_8$. 
\smallskip

For the converse, if there is no $2$-factor containing the intersection of the paths there is no $2$-factor containing either of them. If any subpaths of length 3 of $P_1$ and $P_2$ are in the position $\mathcal{P}_1$, by Theorem~\ref{thm5} $G$ does not have a $2$-factor containing either of $P_1$ and $P_2$. If the paths $P_1$ and $P_2$ are in the position $\mathcal{P}_j$, for $j\in\{2,3,4,5,6,7,8\}$, then we show that $G-E(P_1)$ and $G-E(P_2)$ contain large independent sets, contradicting Theorem~\ref{thm1}. Consider the vertex colouring depicted in Figure~\ref{fig34}. Note that $G-E(P_1)$ contains a large independent set formed by white vertices if we recolour the vertices with label $1$ to the opposite colour. Similarly, if we recolour the vertices with label $2$ to the opposite colour, we obtain a large independent set in $G-E(P_2)$ formed by white vertices in the positions  $\mathcal{P}_2$, $\mathcal{P}_3$, $\mathcal{P}_9$, $\mathcal{P}_{10}$, $\mathcal{P}_{11}$, $\mathcal{P}_{12}$, and $\mathcal{P}_{13}$ and by black vertices in the positions  $\mathcal{P}_4$, $\mathcal{P}_5$, $\mathcal{P}_6$, $\mathcal{P}_7$, and $\mathcal{P}_8$.
`\end{proof}

\noindent {\bf Proof of Theorem~\ref{cor7}:}  Let $V_{2^-}$ and $V_3$ be the sets of vertices in distance at most $2$ ($v\in V_{2^-}$), and $3$ from $v$, respectively. Each edge in $\partial_G(V_{2^-} \cup V_3)$ is contained in an unique $4$-path starting at $v$. We call such paths \emph{leaving paths}.

Consider there exists a pair of distinct leaving paths $P_1$, $P_2$. If one can find a $2$-factor containing one of them, then this is the desired $2$-factor. Thus we may apply Theorem~\ref{thm3}.

If there is no $2$-factor containing $E(P_1) \cap E(P_2)$, then due to Theorem~\ref{plesnik} $|E(P_1) \cap E(P_2)|=3$.
But then $G-(E(P_1) \cap E(P_2))$ is bipartite. Let $v_2$ and $v_3$ be the neighbours of $v$ not contained in $P_1$. But then due to Theorem~\ref{plesnik} there is a $2$-factor containing $v_2vv_3$ and the circuit containing $v$ is either even or contains an edge from $E(P_1) \cap E(P_2)$ in which case its length is more than $7$. Thus $P_1$ and $P_2$ are either in position $\mathcal{P}_9$ or $\mathcal{P}_{12}$ (no other position allows two paths to have the same endpoint).

\smallskip
\noindent
\textbf{Claim 1.} 
For each pair $P_1$, $P_2$ of distinct leaving paths $|E(P_1)\cap E(P_2)| \in \{0, 2\}$. 
\smallskip

If there are at least $7$ distinct leaving paths, then at least $3$ of them share the first edge and thus they share also the second edge (Claim~1). Then at least $2$ of them share the third edge which contradicts Claim~1. Thus there are at most $6$ distinct leaving paths. If some two leaving paths are in position $\mathcal{P}_9$ the imbalance in the partition sizes among vertices in $V_3$ implies that there are at least $10$ distinct leaving paths, a contradiction.

\smallskip
\noindent
\textbf{Claim 2.} 
For each pair $P_1$, $P_2$ of distinct leaving paths $|E(P_1)\cap E(P_2)| = 0$. 
\smallskip

As girth of $G$ is at least $7$ and there are no Moore graphs of degree $3$ and girth $7$ \cite{moore} there are some leaving paths and as cubic graphs have even number of vertices and due to cyclic connectivity there are at least $4$ of them.
Thus at least two of them share the first edge which contradicts Claim~2.


\section*{Acknowledgements}
\noindent The first author acknowledges support from the
VEGA grant No. 1/0813/18, and from the projects APVV-15-0220 and APVV–19–0308. 
The work of the second author was partially supported by the
project GA17-04611S of the Czech Science Foundation and by the project
LO1506 of the Czech Ministry of Education, Youth and Sports.


\end{document}